\documentclass[a4paper,10pt]{article}
\usepackage[utf8]{inputenc}
\usepackage{amsfonts, amssymb, amsmath, amsthm}
\usepackage{mathtools}
\usepackage{enumerate}
\usepackage{graphicx}
\usepackage{fullpage}
\usepackage{verbatim}
\usepackage{hyperref}
\usepackage{authblk}
\hypersetup{colorlinks=true, linkcolor=red, urlcolor=blue, citecolor=blue, pdfpagemode=UseNone, pdfstartview=FitH, bookmarksopen=true}
\hypersetup{pdftitle=Title}

\newtheorem*{theorem*}{Theorem}
\newtheorem{theorem}{Theorem}

\newtheorem*{lemma*}{Lemma}
\newtheorem{lemma}[theorem]{Lemma}

\newtheorem*{proposition*}{Proposition}
\newtheorem{proposition}[theorem]{Proposition}
\newtheorem*{fact*}{Fact}

\newtheorem*{question*}{Question}

\newtheorem*{corollary*}{Corollary}
\newtheorem{corollary}[theorem]{Corollary}

\numberwithin{claimcounter}{theorem}
\newtheorem*{claim*}{Claim}

\theoremstyle{remark}
\newtheorem*{remark*}{Remark}

\theoremstyle{definition}
\newtheorem*{definition*}{Definition}

\numberwithin{theorem}{section}

\newtheorem*{observation*}{Observation}

\title{Weak saturation of multipartite hypergraphs}
\author{Denys Bulavka\thanks{Department of Applied Mathematics, Faculty of Mathematics and
		Physics, Charles University, Prague,
		Czech Republic.}
	\thanks{Supported by
    Charles University project PRIMUS/21/SCI/014, Grant Schemes at CU,
    reg.~no. CZ.02.2.69/0.0/0.0/19\_073/0016935 and by the grant SVV–2020–260578. Email:
    \mbox{dbulavka$@$kam.mff.cuni.cz}.} , Martin Tancer$^*$\thanks{Supported by
      GA\v{C}R grant 22-19073S. Email: \mbox{`MySurname'@kam.mff.cuni.cz}.} , and
    Mykhaylo Tyomkyn$^*$\thanks{Supported by
    GA\v{C}R grant 22-19073S, ERC Synergy Grant DYNASNET 810115 and the
    H2020-MSCA-RISE Project CoSP-GA No. 823748. Email: \mbox{tyomkyn$@$`SameAsMyCoauthors'}.}}

\date{}

\newcommand{\R}{\mathbb{R}}

\DeclareMathOperator{\sgn}{sgn}

\DeclareMathOperator{\supp}{supp}

\DeclareMathOperator{\SPAN}{span}

\DeclareMathOperator{\wsat}{wsat}
\DeclareMathOperator{\mwsat}{w}

\DeclareMathOperator{\rank}{rank}

\newcommand{\rvec}{{\mathbf r}}

\newcommand{\nvec}{{\mathbf n}}

\newcommand{\lip}{{\llcorner}}

\begin{document}

\maketitle

\begin{abstract}
Given $q$-uniform hypergraphs ($q$-graphs) $F,G$ and $H$, where $G$ is a spanning subgraph of $F$, $G$ is called \emph{weakly $H$-saturated} in $F$ if the edges in $E(F)\setminus E(G)$ admit an ordering $e_1,\dots, e_k$ so that for all $i\in [k]$ the hypergraph $G\cup
\{e_1,\dots,e_i\}$ contains an isomorphic copy of $H$ which in turn contains the edge $e_i$. The \emph{weak saturation number} of $H$ in $F$ is the smallest size of an $H$-weakly saturated subgraph of $F$. Weak saturation was introduced by Bollob\'as in 1968, but despite decades of study our understanding of it is still limited. The main difficulty lies in proving lower bounds on weak saturation numbers, which typically withstands combinatorial methods and requires arguments of algebraic or geometrical nature.

In our main contribution in this paper we determine \emph{exactly} the weak saturation number of complete multipartite $q$-graphs in the directed setting, for any choice of parameters. This generalizes a theorem of Alon from 1985. Our proof combines the exterior algebra approach from the works of Kalai with the use of the colorful exterior algebra motivated by the recent work of Bulavka, Goodarzi and Tancer on the colorful fractional Helly theorem. In our second contribution answering a question of Kronenberg, Martins and Morrison, we establish a link between weak saturation numbers of bipartite graphs in the clique versus in a complete bipartite host graph. In a similar fashion we asymptotically determine the weak saturation number of any complete $q$-partite $q$-graph in the clique, generalizing another result of Kronenberg et al.
\end{abstract}	

\section{Introduction}

Let $F$ and $H$ be $q$-uniform hypergraphs ($q$-graphs for short); we identify
hypergraphs with their edge sets. We say that a subgraph $G\subseteq F$ is
\emph{weakly $H$-saturated} in $F$ if the edges of $F\setminus G$ can be ordered
as $e_1,\dots,e_k$ such that for all $i\in [k]$ the hypergraph $G\cup
\{e_1,\dots,e_i\}$ contains an isomorphic copy of $H$ which in turn contains the
edge $e_i$. We call such $e_1,\dots, e_k$ an \emph{$H$-saturating sequence} of $G$ in
$F$. The \emph{weak saturation number} of $H$ in $F$, $\wsat(F,H)$ is the
minimum number of edges in a weakly $H$-saturated subgraph of $F$. When $F$ is
complete of order $n$, we simply write $\wsat(n,H)$.

Weak saturation was introduced by Bollob\'as~\cite{Bollobas} in 1968 and is
related to (strong) graph saturation: $G$ is $H$-saturated in $F$ if adding any
edge of $F\setminus G$ would create a new copy of $H$. However, a number of
properties of weak saturation make it a more natural object of study. Firstly,
it follows from the definition that any graph $G$ achieving $\wsat(F,H)$ has to
be $H$-free (we could otherwise remove an edge from a copy of $H$ in $G$
resulting in a smaller example), while for strong saturation $H$-freeness may or
may not be imposed, resulting in two competing notions (see~\cite{MoshkShap} for
a discussion). Secondly, a short subadditivity argument originally due to
Alon~\cite{Alon} shows that for every 2-uniform $H$,   $\lim_{n\rightarrow \infty}\wsat(n,H)/n$ exists. Whether the same holds for strong saturation is a longstanding conjecture of Tuza~\cite{Tuza2}. And thirdly, weak saturation lends itself to be studied via algebraic methods, thus offering insight into algebraic and matroid structures underlying graphs and hypergraphs.
 
The most natural case when $F$ and $H$ are cliques was the first to be studied.
Let $K_r^q$ denote the complete $q$-graph of order $r$.  
Confirming a conjecture of Bollob\'as, Frankl~\cite{Frankl82}, and
Kalai~\cite{Kal84, Kal85} independently proved that
$\wsat(n,K_r^{q})=\binom{n}{q}-\binom{n-r+q}{q}$. Another proof has been given
by Alon~\cite{Alon} and in hindsight this conjecture could be also derived from
an earlier paper of~Lov\'asz~\cite{Lovasz77}.
While the upper bound is a construction that is easy to guess (a common
feature in weak saturation problems), all of the above lower bound proofs rely
on algebraic or geometric methods, and no purely combinatorial proof is known
to this date. 

In the subsequent years weak saturation has been studied extensively~\cite{Alon,
  Tuza, Erdos91, Pikhurko01a, Tuza88, MoshkShap, Pikhurko01b, Semanivsin97,
  Borowiecki02, Sidorowicz07, Faudree14, BBMR, Balogh98, Morrison18}. Despite this, our understanding of weak saturation numbers is still rather limited. For instance we do not know whether for $q\geq 3$ we have a similar limiting behavior as in the graph case, in that $\lim_{n\rightarrow \infty}\wsat(n,H)/n^{q-1}$ always exists; this has been conjectured by Tuza~\cite{Tuza}.

In this paper we address the case when $H=K^q_{r_1,\dots,r_d}$ is a complete
$d$-partite $q$-graph for arbitrary $d\geq q > 1$. That is, $V(H)$ is a disjoint union of sets $R_1,\dots,R_d$ with $|R_i|=r_i$ and 
$$E(H)=\left\{e\in \binom{V(H)}{q}\colon |e\cap R_i|\leq 1 \text{ for all } i
\in [d]\right\},$$ 
in particular, for $q=2$ we recover the usual complete multipartite graphs. This is perhaps the next most natural class of hypergraphs to consider after the cliques.

For the host graph $F$, besides the clique it is natural to consider a larger complete $d$-partite $q$-graph $K^{q}_{n_1,\dots, n_d}$. In the latter case we have a choice between the \emph{undirected} and \emph{directed} versions of the problem. The former follows the definition of weak saturation given at the beginning, while in the latter we additionally impose that the new copies of $H$ in $F$ created in every step ``point the same way'', i.e. have $r_i$ vertices in the $i$-th partition class for all $i\in[d]$ (see below for a formal definition).

All three above versions have been studied in the past. For $q=2$,
Kalai~\cite{Kal85} determined $\wsat(n,K_{r,r})$ for large enough $n$. Kronenberg, Martins and
Morrison~\cite{KMM} recently extended it to
$\wsat(n,K_{r,r-1})$ and asymptotically to all $\wsat(n,K_{s,t})$. No other
values $\wsat(n,K^{q}_{r_1,\dots, r_d})$ are known except for $r_1=\dots=r_d=1$
when $H$ is a clique and a handful of closely related cases, e.g., when all
$r_i$ but one are $1$~\cite{Pikhurko01b}.  When both $H$ and $F$ are complete
$d$-partite, for $d=q$ Alon~\cite{Alon} solved the problem in the directed
setting. Moshkovitz and Shapira~\cite{MoshkShap}, building on Alon's work,
settled the undirected case, determining $\wsat(K^d_{n_1,\dots ,n_d},
K^d_{r_1,\dots ,r_d})$. There has been no progress for $d>q$.

In our main contribution in this paper we settle completely the directed case for all $q$ and $d$. To state the problem formally, let $\rvec=(r_1,\dots,r_d)$ and $\nvec = (n_1,\dots, n_d)$
be integer vectors such that $1\leq r_i\leq n_i$. Suppose $N=N_1\sqcup \dots \sqcup N_d$ where $|N_i|=n_i$ and $\sqcup$ denotes a disjoint union. 
Let $K^q_{\nvec}$ be the
complete $d$-partite $q$-graph on $N$ whose partition classes are the $N_i$, and let 
$K^q_{\rvec}$ be an unspecified complete $d$-partite $q$-graph on the same partition classes, with $r_i$ vertices in each $N_i$. 
Given a subgraph $G$ of $K^q_{\nvec}$, a sequence of edges $e_1,\dots
,e_k$ in $K^q_{\nvec}$ is a \emph{(directed) $K^q_{\rvec}$-saturating sequence of $G$ in
	$K^q_{\nvec}$} if: (i) $K^q_{\nvec}\setminus G = \{e_1,\dots ,e_k\}$; (ii) for every
$j\in[k]$ there exists $H_j \subseteq G \cup \{e_1,\dots, e_j\}$ isomorphic to 
$K^q_{\rvec}$ such that $e_j \in H_j$ and $|
V(H_j)\cap N_i| = r_i$ for all $i\in[d]$. The $q$-graph $G$ is said to be \emph{(directed) weakly
	$K^q_{\rvec}$-saturated in $K^q_{\nvec}$} if it admits a $K^q_{\rvec}$-saturating sequence in the latter.
The \emph{(directed) weak saturation number of
	$K^q_{\rvec}$ in $K^q_{\nvec}$}, in notation $\mwsat(K^q_{\nvec}, K^q_{\rvec})$, is the minimal number of edges in a weakly $K^q_{\rvec}$-saturated subgraph of $K^q_{\nvec}$.
\begin{theorem}\label{t:mwsat} 
	For all $d\geq q\geq 2$, $\nvec$ and $\rvec$ we have
	$$\mwsat (K^q_{\nvec}, K^q_{\rvec}) = \sum_{I\in
		\binom{[d]}{q}}\prod_{i\in I}n_i - \sum_{I\in
		\binom{[d]}{\leq q}} \prod_{i\in I}(n_i-r_i) .$$ 
\end{theorem}
\noindent
In the above formula 
$\binom{[d]}{\leq q}$ stands for the
set of all subsets of $[d]$ of size at most $q$, and we use the convention that $\prod_{i\in \emptyset}(n_i - r_i) = 1$.  

As mentioned, the $d=q$ case of Theorem~\ref{t:mwsat} was proved by Alon~\cite{Alon}. Hence our result generalizes Alon's theorem to arbitrary $d\geq q$. When $H$ is balanced, that is when $r_1=\dots = r_d$, there is no difference between the directed and undirected partite settings. Writing $K^{q}(r;d)$ for $K^{q}_{r,\dots, r}$ ($d$ times), Theorem~\ref{t:mwsat} thus determines the weak saturation number of $K^{q}(r;d)$ in complete $d$-partite $q$-graphs.
\begin{corollary}\label{c:unordered}
	For all $d\geq q\geq 2$ and $n_1, \dots, n_d \geq r\geq 1$ we have
	$$\wsat(K^{q}_{n_1,\dots, n_d}, K^{q}(r;d))= \sum_{I\in
		\binom{[d]}{q}}\prod_{i\in I}n_i - \sum_{I\in
		\binom{[d]}{\leq q}} \prod_{i\in I}(n_i-r) .
	$$
\end{corollary}
Our proof of Theorem~\ref{t:mwsat} combines exterior algebra techniques in the spirit of~\cite{Kal85} with a new ingredient: the use of the colorful exterior algebra inspired by the recent work of Bulavka, Goodarzi and Tancer on the colorful fractional Helly theorem~\cite{bulavka-goodarzi-tancer21socg}.

Kronenberg, Martins and Morrison (\cite{KMM}, Section 5) remarked that while the
values $\wsat(n,K_{t,t})$ and $\wsat(K_{\ell,m},K_{t,t})$ for $\ell+m=n$, which
were determined in  separate works, are of the same order of magnitude, it is
not obvious if there is any direct connection. In our second contribution in
this paper we establish such a connection using a tensoring trick. As we have mentioned earlier,
$2$-graphs $H$ satisfy $\wsat(n,H)=c_H n + o(n)$, and Alon's proof of this
fact~\cite{Alon} can be straightforwardly adjusted to show that $\wsat(K_{n,n},H)=c'_H\cdot
2n+o(n)$ when $H$ is bipartite. We show that in fact $c_H=c'_H$. A minor adjustment to our proof gives that, for any rational $0<\alpha<1$, the quantities $\wsat(n,H)$ and $\wsat(K_{\alpha n, (1-\alpha)n}, H)$, when $\alpha n\in \mathbb{Z}$, are of the same order of magnitude. Setting $H=K_{t,t}$ answers the above question of~\cite{KMM}.

For $q\geq 3$ while we
do not have (yet) the same knowledge of limiting constants, a similar method determines asymptotically the weak saturation number of complete $d$-partite $d$-graphs in the clique, generalizing Theorem 4 of~\cite{KMM}.

\begin{theorem}\label{t:tensor}
For every bipartite $2$-uniform graph $H$ we have
	\begin{equation}\label{e:limconst}
		\lim_{n\rightarrow \infty}\frac{\wsat(n,H)}{n}=\lim_{n\rightarrow \infty}\frac{\wsat(K_{n,n},H)}{2n}.
	\end{equation}
 \noindent
Furthermore, for any $d\geq 2$ and $1\leq r_1\leq\dots \leq r_d$ we have
	\begin{equation}\label{e:hypreduction}
		\wsat(n,K^{d}_{r_1,\dots,r_d})=\frac{r_1-1}{(d-1)!}n^{d-1}+O(n^{d-2}).
	\end{equation}	
\end{theorem} 
The rest of the paper is organized as follows. In Section~\ref{sec:upper} we give a construction for the upper bound in Theorem~\ref{t:mwsat}. In Section~\ref{sec:algebra} we review the algebraic tools, setting the stage for the lower bound proof in Section~\ref{sec:lower}. In Section~\ref{sec:clique} we discuss weak saturation in the clique and prove Theorem~\ref{t:tensor}.

\paragraph*{Notation.} As usual, $[n]$ abbreviates the set $\{1,\dots,n\}$. The
symbol $\sqcup$ denotes a disjoint union of sets. For a set $M$ and integer $q\geq 0$, $\binom{M}{q}$ and $\binom{M}{\leq q}$ denote the set of all subsets of $M$ of size exactly $q$ and of size most $q$, respectively. We use $\pm$ to denote an unspecified factor of either $+1$ or $-1$. 

$K_n^q$ denotes the complete $q$-uniform hypergraph ($q$-graph) of order $n$. When the vertex set of the said $q$-graph is $[n]$, we write $K_{[n]}^q$. The complete $d$-partite $q$-graph with $n_i$ vertices in the $i$-th partition class is denoted by $K^q_{n_1,\dots,n_d}$; when $n_1=\dots=n_d=n$ we write simply $K^q(n;d)$.

Note that in Sections~\ref{sec:upper}--\ref{sec:lower} we work solely in the directed partite setup (Theorem~\ref{t:mwsat}), while in  Section~\ref{sec:clique} we deal with the undirected partite and the clique setups (Theorem~\ref{t:tensor}). 
In the directed setup our $q$-graphs are defined on a vertex set $N$ of size
$n$ with a fixed $d$-partition $N=N_1\sqcup \dots \sqcup N_d$, where
$|N_i|=n_i$ for all $i\in [d]$. Consequently, we use $K^q_{\nvec}$ to denote
the complete $d$-partite $q$-graph on $N$ with respect to this partition. (Up
to a graph isomorphism, $K^q_{\nvec}$ is uniquely determined by $q$ and
$\nvec$, thus we do not display $N$ in the notation.)
For any $M\subseteq N$ the induced subgraph of $K^q_{\nvec}$ on $M$ is denoted
by $K^q_{\nvec}[M]$. The directed weak saturation number defined above is denoted by
$\mwsat (K^q_{\nvec}, K^q_{\rvec})$, as opposed to
$\wsat(K^q_{n_1,\dots,n_d},K^q_{r_1,\dots,r_d})$ in the undirected setting, a
similar notation was employed in~\cite{KMM}.

\section{Theorem~\ref{t:mwsat}: the upper bound}\label{sec:upper}
In this section we prove the upper bound in Theorem~\ref{t:mwsat} by exhibiting a weakly $K^q_{\rvec}$-saturated $q$-graph $G$. Fix a subset $R \subseteq
N$ such that $|R\cap N_i| = r_i$ for every $i\in [d]$ and set
\begin{equation*}
  \Sigma := \Big \{S\in \binom{N\setminus R}{\leq q}\colon |S\cap N_i|\leq 1 \text{ for each }i\in [d]\Big \}.
\end{equation*}
We define $G$ via its complement in $K^q_{\nvec}$ as follows. For every $S\in
\Sigma$ choose an edge $\lambda(S)\in K^q_{\nvec}[R\cup S]$ satisfying $S\subseteq \lambda(S)$. Note that the assignment $\lambda$ is injective, as $\lambda(S)\cap (N \setminus R)=S$.
Recall that we associate hypergraphs with their edge sets. Define $$G:=K^q_{\nvec}\setminus \bigcup_{S\in\Sigma}\lambda(S),$$
so that
\begin{equation*}
  |E(G)| = \sum_{I\in
    \binom{[d]}{q}}\prod_{i\in I}n_i - \sum_{I\in \binom{[d]}{\leq q}} \prod_{i\in I}(n_i-r_i).
\end{equation*}

Notice that the choices of $\lambda(S)$ are not unique, but as the next lemma shows, each of them yields a weakly $K^q_{\rvec}$-saturated $q$-graph. Such non-uniqueness is a common occurrence in weak saturation: for instance, every $n$-vertex tree is an extremal example for weak triangle saturation in $K_n$.

\begin{lemma}
  \label{l:gubound}
 The $q$-graph $G$ defined above is weakly 
  $K^q_{\rvec}$-saturated. Therefore,
  $$\mwsat(K^q_{\nvec}, K^q_{\rvec}) \leq |E(G)|= \sum_{I\in
    \binom{[d]}{q}}\prod_{i\in I}n_i - \sum_{I\in \binom{[d]}{\leq q}} \prod_{i\in I}(n_i-r_i).$$
\end{lemma}
\begin{proof}
  For each $0\leq k\leq q$ let 
  $$G_k := G \cup \{T\in K^q_{\nvec}\colon |T\setminus R| \leq k\},$$
  and put $G_{-1}:=G$. 
  We claim that adding
  any new edge $L\in K^q_{\nvec}$ with $|L\setminus R| = k$  to $G_{k-1}$ creates a new copy
  of $K^q_{\rvec}$ containing $L$. This gives rise to a $K^q_{\rvec}$-saturating sequence between $G_{k-1}$ and $G_k$ and, by extension, between
  $G=G_{-1}$ and $G_q=K^q_{\nvec}$.
  
  First, notice that $G_0$ is obtained from $G_{-1}$ by adding the sole missing
  edge $\lambda(\emptyset)$. Doing so creates a new copy of $K^q_{\rvec}$,
  namely $K^q_{\nvec}[R]$. For an arbitrary  $k$, suppose that $L$ is a missing edge in
  $G_{k-1}$ such that $S:= L\setminus R$ is of size $k$. Observe that every $T\in K^q [R \cup S]$ is an edge in
  $G_{k-1}$ unless $T=L$ . Indeed, if $|T\setminus R|<k$ then this holds by
  definition of $G_{k-1}$. While otherwise we have $T\setminus R=S$. Hence, by
  the definition of $G$, we have $L=\lambda(S)$, so that either $T=L$ or $T\in
  G\subseteq G_{k-1}$. Therefore, adding $L$ to $G_{k-1}$ creates a new copy of
  $K^q_{\nvec}[R\cup S]$ containing $L$ and a fortiori also a new copy of $K^q_{\rvec}$ containing $L$, as desired. 
\end{proof}

\section{Algebraic background}\label{sec:algebra}
In this section we introduce the linear algebra tools needed for the proof of
the lower bound in Theorem~\ref{t:mwsat}. In Sections~\ref{subsec:exterior}
and~\ref{subsec:lip} we largely follow~\cite[Sec.~2]{kalai84intersection}
though we sometimes provide more detail. (For comparison~\cite{Kal85}
works with a dual generic basis. We believe that the difference is
not essential.) In Section~\ref{subsec:col} we loosely
follow~\cite{bulavka-goodarzi-tancer21socg}.

Before we start explaining the algebraic background, we will try to sketch why
algebraic tools can be useful in this context.
This sketch should be understood loosely---we do not provide any guarantees for
the claims in this sketch. In particular, many important technical details are
skipped in the sketch. Understanding this sketch is not required in the
following text, thus it can be skipped.

Consider first the somewhat trivial case of providing the lower bound on
$\wsat(n,K_3)$, the weak saturation number of the complete graph $K_3$ in
$K_n$. Consider a subgraph $G$ of $K_n$ and a saturating sequence $e_1, \dots
e_k$ of edges in $E(K_n) \setminus E(G)$. Let $G_i := G \cup \{e_1, \dots,
e_i\}$. Because the sequence is saturating, we know that $G_i$ contains a copy of $K_3$
containing $e_i$. This means that the dimension of the cycle space of $G_i$ is
strictly larger than the dimension of the cycle space of $G_{i-1}$. Because the
final dimension of the cycle space of $K_n$ equals $\binom{n-1}2$, we may
perform at most $\binom{n-1}2$ such steps. In other words $k \leq \binom{n-1}2$
and thus $|E(G)| \geq \binom n2 - \binom{n-1}2$ as required. 

In the language of algebraic topology (which we however do not use in the
proofs, no topological background is required), the property that the dimension
of the cycle space increases can be phrased so that a new copy of $K_3$ in each
step belongs to the kernel of the standard boundary operator. For more
complicated (hyper)graphs than $K_3$ it is actually useful to use several
independent boundary operators in order to generalize the aforementioned
approach. Using such independent operators can be actually efficiently phrased
in terms of exterior algebra (without mentioning algebraic topology). They
correspond to the \emph{left interior product}, which we will discuss later
on, subject to some suitable independence (genericity) condition.\footnote{Perhaps the closest relation between the boundary operators and
the left interior product can be seen in Lemma~\ref{l:fSLeR} interpreting $e_R$
as a simplex with set of vertices $R$, and $f_T\lip$ as an operator removing $t$ times the top-dimensional simplices, yielding a linear combination of simplices
$f_S$ with $r-t$ vertices. (However, for this relation, 
it would be even better to express the right hand side using $e_S$ so that all
possible $e_S$ would appear.) 
Adding a colorful aspect (in our case) then makes it easier to work
with multipartite (hyper)graphs rather than complete ones.}

\subsection{Exterior algebra.}\label{subsec:exterior} 
Let $N$ be a set of size $n$, ordered with
a total order $<$. Later on the elements of $N$ will represent vertices of a $q$-graph and we will typically denote them by letters such as
$v$ or $w$. Let $V$ be an $n$-dimensional real vector space with a basis $(e_v)_{v \in N}$.
The \emph{exterior algebra} of $V$, denoted by $\bigwedge V$, is a
$2^n$-dimensional vector space with basis $(e_S)_{S\subseteq N}$ and an
associative 
bilinear product operation, denoted by $\wedge$, that satisfies 
\begin{enumerate}[(i)]
  \item $e_{\emptyset}$ is the
    neutral element, i.e. $e_{\emptyset}\wedge e_S = e_S = e_S \wedge
    e_{\emptyset}$; 
 \item $e_S = e_{s_1} \wedge
\cdots \wedge e_{s_k}$ for $S = \{s_1 < \cdots < s_k\} \subseteq N$; 
 \item $e_v
\wedge e_w = - e_w\wedge e_v$ for all $v,w \in N$. 
 \end{enumerate}
For $0\leq k\leq n$ we denote by $\bigwedge^k V$ the subspace of $\bigwedge V$ with basis $(e_S)_{S\in \binom{N}{k}}$. 
Denote by $\langle \cdot, \cdot \rangle$ the standard inner product (dot
product) on $V$ as well as on $\bigwedge V$ with respect to the basis
$(e_v)_{v\in N}$ and $(e_S)_{S\subseteq N}$ respectively; that is, for every
pair of sets $S,T\subseteq
N$, the inner product $\langle e_S, e_T \rangle$ is $1$ if $S=T$ and $0$ otherwise.

If $(f_v)_{v \in N}$ is another basis of $V$,
then $(f_S)_{S\subseteq N}$ is a new basis of $\bigwedge V$, where $f_S$
stands for $f_{s_1} \wedge \cdots \wedge f_{s_k}$ for $S = \{s_1 < \cdots <
s_k\} \subseteq N$. Similarly, $(f_S)_{S\in \binom{N}k}$ is a basis of
$\bigwedge^k V$ for $k \in \{0, \dots, n\}$. The formulas (i), (ii) and
(iii) remain valid for the basis $(f_v)_{v \in N}$ due to definition of $f_S$
and bilinearity of $\wedge$.
In particular, $\bigwedge V$ and $\bigwedge^k V$ do not depend on the
initial choice of the basis.
Using (ii) and (iii)
iteratively, for $S, T \subseteq N$ we get
\begin{equation}
  \label{e:sign_wedge_new}
  f_S \wedge f_T = \begin{cases} 
    \sgn(S,T) f_{S\cup T} &\hbox{ if } S \cap T = \emptyset\\
    0 &\hbox{ if } S \cap T \neq \emptyset,\\
  \end{cases}
\end{equation}
where $\sgn(S,T)$ is the sign of the permutation of $S \cup T$ obtained by
first placing the elements of $S$ (in our total order $<$) and then the elements of $T$. Equivalently, $\sgn(S,T) = (-1)^{\alpha(S,T)}$ where
$\alpha(S,T) = |\{(s,t) \in S\times T \colon  t < s\}|$ is the number of
transpositions.

As a consequence we obtain the following useful formula. Let $M_1,
\dots, M_\ell$ be pairwise disjoint subsets of $N$ and $s_1, \dots, s_\ell$ be
integers with $0 \leq s_i \leq |M_i|$. Suppose that for each $i\in [\ell]$ we are given $$h_i = \sum_{S_i \in
\binom{M_i}{s_i}} \lambda_{S_i} f_{S_i}$$ 
for $\lambda_{S_i} \in \R$ (so that $h_i \in \bigwedge^{s_i} V$).
Then by bilinearity of $\wedge$ and~\eqref{e:sign_wedge_new} we get
\begin{equation}
  \label{e:wedge_multilinearity_new}
  h_1\wedge \cdots \wedge h_{\ell} 
  = \sum_{\substack{(S_1,\dots ,S_\ell) \in \\ \binom{M_1}{s_1} \times \dots \times
  \binom{M_\ell}{s_\ell}}} \left 
  (\prod_{i\in [\ell]} \lambda_{S_i} \right ) f_{S_1}\wedge \cdots \wedge
  f_{S_\ell} 
  = \sum_{\substack{(S_1,\dots ,S_\ell) \in \\ \binom{M_1}{s_1} \times \dots \times
  \binom{M_\ell}{s_\ell}}} 
  \pm \left (\prod_{i\in [\ell]} \lambda_{S_i} \right ) f_{S_1 \cup \cdots \cup
  S_{\ell}}.
\end{equation}

Let $A = (a_{vw})_{v,w \in N}$
be the transition matrix from $(e_v)_{v\in N}$ to $(f_v)_{v \in N}$, meaning that $f_v = \sum_{w \in N} a_{vw}e_w$. Then, for $S \subseteq N$ of size
$k$, $f_S$ can be expressed as
\begin{equation}
  \label{e:f_via_e_d}
  f_S = \sum\limits_{T \in \binom{N}{k}} \det (A_{S|T}) e_T,
\end{equation}
where $A_{S|T}$ is the submatrix of $A$ formed by rows in $S$ and columns in
$T$, i.e. $A_{S|T} = (a_{vw})_{v\in S, w\in T}$.

As noted in~\cite{kalai84intersection}, it follows from the Cauchy-Binet
formula that if the basis $(f_v)_{v\in N}$ is orthonormal then
  $(f_S)_{S\subseteq N}$ is orthonormal as well. For completeness, we provide a
  short explanation.
Let $S,L\subseteq
  N$ be a pair of subsets. If $|S| \neq |L|$, then
  $f_S$ and $f_L$ belong to two orthogonal subspaces of $\bigwedge V$, namely $\bigwedge^{|S|} V$ and
  $\bigwedge^{|L|} V$, and so $\langle f_S, f_L \rangle = 0$. On the other
  hand, if $|S|=|L|=:k$, then by writing $f_S$ and $f_L$ in the standard basis $(e_T)_{T\subseteq N}$
  we have that $$\langle f_S, f_L \rangle = \sum_{T\in \binom{N}{k}} \det
  (A_{S|T}) \det (A^t_{L|T}) = \det (A_{S|N} A^t_{L|N}),$$ where $B^t$ stands for the 
  transpose matrix of $B$ (and expressions like $A^t_{L|T}$ stand for
  $(A_{L|T})^t$), and the last equality holds by the Cauchy-Binet formula
  (see e.g. Section 1.2.4 of~\cite{Gantmacher64}). Notice that for any $u\in S$
  and $w\in L$ we have $(A_{S|N} A^t_{L|N})_{u,w} = \langle f_u, f_w \rangle$, and since $(f_v)_{v\in
    N}$ is orthonormal this is $1$ if $u=w$ and
  $0$ otherwise. Therefore, if $S=L$, the product $A_{S|N}A^t_{L|N}$ is the identity
  matrix and consequently the determinant will be $1$. On the other hand, if $S\neq L$, the product $A_{S|N}A^t_{L|N}$ will have a zero column, and so the
  determinant will be $0$. The above claim follows.

We say that the change of
basis from $(e_v)_{v \in N}$ to $(f_v)_{v \in N}$
is \emph{generic} if $\det (A_{S|T})
\neq 0$ for every $S, T \subseteq N$ of the same size; that is, every square
submatrix of $A$ has full rank. It is known (see e.g.~\cite{kalai84intersection}) that
$(f_v)_{v\in N}$ can be chosen to be both generic and orthonormal. For a basis $(f_v)_{v\in N}$ generic with respect to $(e_v)_{v\in N}$ and a
pair of sets $S,T\in \binom{N}{k}$ we have 
\begin{equation}
  \label{e:fS_eT}
  \langle f_S, e_T \rangle \stackrel{\eqref{e:f_via_e_d}}{=} \langle \sum_{T'\in \binom{N}{k}} \det(A_{S|T'})  e_{T'},e_T \rangle = \sum_{T'\in \binom{N}{k}} \det(A_{S|T'}) \langle e_{T'},e_T \rangle= \det A_{S|T} \neq 0.
\end{equation}
\subsection{Left interior product.}
\label{subsec:lip}
The following lemma defines $g \lip f$, the \emph{left interior product} of $g$
and $f$. We refer to Section 2.2.6 of \cite{Rosen19} for a
more extensive coverage of the topic.
\begin{lemma}
  \label{l:lip_existence_uniqueness}
  For any $f,g \in \bigwedge V$ there exists a unique element $g\lip f \in \bigwedge V$
  that satisfies
  \begin{equation} \label{e:lip_definition}
  	\langle h, g \lip f\rangle = \langle h \wedge g, f\rangle \text{ for all } h \in \bigwedge V.
  \end{equation}
  Furthermore, assuming $f\in \bigwedge^s V$ and $g\in \bigwedge^t V$, if $t > s$ then $g\lip f =0$, while if $t\leq s$ then $g\lip f \in \bigwedge^{s-t}V$.
\end{lemma}
\begin{proof}
For $f, g\in \bigwedge V$ we set 
\[g\lip f := \sum_{S\subseteq N}\langle
  e_S\wedge g, f \rangle e_S.\]
To verify that this satisfies~\eqref{e:lip_definition} let
$h\in \bigwedge V$ be arbitrary. By bilinearity of $\langle \cdot, \cdot
\rangle$ and $\wedge$, and orthonormality of  $(e_S)_{S\subseteq N}$ we have
\begin{align*}
  \langle h, g\lip f \rangle
  &= \langle h, \sum_{S\subseteq N}\langle e_S\wedge g, f \rangle e_S \rangle 
  =  \sum_{S\subseteq N} \langle e_S\wedge g, f \rangle \langle h, e_S \rangle\\
  &= \big\langle \sum_{S\subseteq N}  \langle h, e_S \rangle (e_S\wedge g), f  \big\rangle  
  = \big\langle \big(\sum_{S\subseteq N} \langle h, e_S \rangle e_S\big) \wedge g, f \big\rangle\\
  &=\langle h\wedge g, f \rangle.
\end{align*}
\noindent
To show uniqueness, suppose that $z$ is an
element in $\bigwedge V$ that satisfies~\eqref{e:lip_definition}. 
Then for each $T \subseteq N$ we have $$\langle e_T, z \rangle \stackrel{\eqref{e:lip_definition}}{=} \langle e_T\wedge g, f
\rangle \stackrel{\eqref{e:lip_definition}}{=} \langle e_T, g\lip f \rangle.$$
Therefore $z$ and $g\lip f$ are identical, as their inner products with all basis elements coincide.

Now assume that $f\in \bigwedge^s V$ and $g\in \bigwedge^t V$, 
and let $S\subseteq N$ be arbitrary. By~\eqref{e:lip_definition} we have
$$\langle e_S, g\lip f \rangle = \langle e_S\wedge g, f  \rangle.$$ Observe that $e_S\wedge g \in \bigwedge^{|S|+t}$ while $f\in \bigwedge^s V$ and these spaces are orthogonal unless $|S|+t=s$. Hence,
$g\lip f = 0$ for $t>s$ and  $g\lip f \in \bigwedge^{s-t}V$ otherwise.
\end{proof}
\noindent
It is straightforward to check from the definition that the left interior product is bilinear:
\begin{itemize}
  \item $(f+g)\lip h = (f\lip h) + (g\lip h)$,
  \item $f\lip(g+h) = (f\lip g) + (f\lip h)$,
\end{itemize}
 and satisfies
\begin{equation}
  \label{e:double_left}
  h\lip (g\lip f) = (h\wedge g)\lip f.
\end{equation}
\noindent
With $\sgn (\cdot, \cdot)$ as defined in Section~\ref{subsec:exterior} we obtain the following statement. 
\begin{lemma}
  \label{l:lip_fTfS}
  Let $(f_v)_{v\in N}$ be an orthonormal basis of $V$. Then, for any $S,T \subseteq N$ we
  have 
  \[
    f_T \lip f_S = 
    \begin{cases} 
      \sgn(S\setminus T, T) f_{S\setminus T} & \text{if } T \subseteq S, \\
      0       & \text{otherwise.} 
    \end{cases}
  \]
\end{lemma}
\begin{proof}
  Put $s := |S|$ and $t := |T|$. If $t>s$ then by
  Lemma~\ref{l:lip_existence_uniqueness} we have $f_T\lip f_S = 0$ and
  the conclusion follows. So we may assume that $s\geq t$, and by the same lemma
  it follows that $f_T\lip f_S \in \bigwedge^{s-t} V$. Since the basis $(f_v)_{v\in N}$ is
orthonormal, so is the basis $(f_L)_{L\in \binom{N}{s-t}}$ of
$\bigwedge^{s-t} V$, as observed in Section~\ref{subsec:exterior}.
Expressing $f_T\lip f_S$ in this basis and using~\eqref{e:lip_definition}, we obtain
  \begin{equation*}
    f_T\lip f_S =
    \sum_{L\in\binom{N}{s-t}}\langle f_L, f_T\lip f_S \rangle  f_L =
    \sum_{L\in \binom{N}{s-t}} \langle f_L \wedge f_T , f_S \rangle f_L.
  \end{equation*}
  Due to~\eqref{e:sign_wedge_new} and orthonormality of $(f_v)_{v\in N}$ we have $\langle f_L \wedge f_T , f_S \rangle=0$ unless $T
  \subseteq S$ and $L = S \setminus T$. Therefore, using~\eqref{e:sign_wedge_new} again we get  
  \[
    f_T \lip f_S = \begin{cases}
      \langle f_{S\setminus T} \wedge f_T, f_S\rangle f_{S\setminus T}  
      =
       \sgn(S \setminus T, T) f_{S\setminus T} & \hbox{ if } T \subseteq S,\\
      0 & \hbox{ if } T \not\subseteq S.\\
    \end{cases}
  \]
\end{proof}

\begin{lemma}
  \label{l:fSLeR}
  Let $(f_v)_{v\in N}$ be a generic orthonormal basis of $V$ with respect to $(e_v)_{v\in N}$. For
  a pair of sets $T,R\subseteq N$ of sizes $t$ and $r$, respectively, such that $r \geq t$ we have 
  \begin{equation*}
    f_T\lip e_R = \sum_{S\in \binom{N\setminus T}{r-t}} \lambda_S f_S,
  \end{equation*}
  where all the coefficients $\lambda_S$ are non-zero.
\end{lemma}
\begin{proof}
  By Lemma~\ref{l:lip_existence_uniqueness} we have that $f_T\lip e_R\in
  \bigwedge ^{r-t} V$. Since $(f_S)_{S\in \binom{N}{r-t}}$ is an orthonormal basis of
  $\bigwedge^{r-t} V$, we can write 
  $$f_T\lip e_R = \sum_{S\in \binom{N}{r-t}} \langle f_S, f_T\lip e_R \rangle
  f_S.$$ 
  Applying~\eqref{e:lip_definition} and~\eqref{e:sign_wedge_new} gives
  $$
  \langle f_S, f_T\lip e_R \rangle = \langle  f_S\wedge f_T, e_R \rangle
  = \begin{cases} \pm \langle f_{S\cup T}, e_R \rangle &\hbox{ if } S\cap T =
  \emptyset, \hbox{ equivalently if } S \in \binom{N\setminus T}{r-t},\\ 0 &\hbox{ otherwise.} \end{cases}
  $$
  Setting $\lambda_S = \langle  f_S\wedge f_T, e_R \rangle$ for $S \in
  \binom{N\setminus T}{r-t}$, we thus obtain 
  $$f_T\lip e_R = \sum_{S\in \binom{N \setminus T}{r-t}} \lambda_S f_S,$$ 
  as
  claimed. In addition, since we assumed that $(f_v)_{v\in N}$ is generic with
  respect to $(e_v)_{v\in N}$, we have $\lambda_S = \pm \langle f_{S\cup T},
  e_R \rangle \neq 0$ by~\eqref{e:fS_eT} for all $S \in \binom{N\setminus
  T}{r-t}$.
\end{proof}

\subsection{Colorful exterior algebra.}\label{subsec:col}
As we are interested in multipartite hypergraphs it is natural to assume
in addition that the set $N$ is partitioned as a disjoint union $N =
N_1 \sqcup N_2 \sqcup \cdots \sqcup N_d$; consistently with the introduction
$n_i := |N_i|$. Here each $N_i$ is ordered by a total
order $<_i$. We extend these orders to the whole $N$ as follows, for $x\in N_i$
and $y\in N_j$, we say that
\begin{equation*}
  x<y \text{ if } i<j \text{ or if } i=j \text{ and } x<_i y.
\end{equation*}
 Given the standard basis $(e_v)_{v
\in N}$ of $V$ we say that a basis $(f_v)_{v \in N}$ is \emph{colorful} with respect
to this partition if $(f_v)_{v \in N_i}$ generates the same subspace of $V =
\R^N$ as $(e_v)_{v \in N_i}$ for every $i \in [d]$; we denote this
subspace $V_i$. Put differently, the 
transition matrix $A$ from $(e_v)_{v \in N}$ to $(f_v)_{v \in N}$ is a
block-diagonal matrix with blocks $N_i \times N_i$ for $i \in [d]$.
We also say that $(f_v)_{v \in N}$ is \emph{colorful generic} (with respect
to this partition) if the basis change from $(e_v)_{v \in N_i}$ to $(f_v)_{v
  \in N_i}$ is generic for every $i \in [d]$. It is possible to choose a basis
  which is simultaneously colorful generic with respect to a given partition and
  orthonormal by choosing each change of basis from $(e_v)_{v \in N_i}$ to $(f_v)_{v
  \in N_i}$ generic and orthonormal. 

By $\bigwedge V_i$ we
  denote the subalgebra of $\bigwedge V$ generated by $e_S$ for $S \subseteq
  N_i$ and by $\bigwedge^k V_i$ the subspace of $\bigwedge V_i$ with basis
  $(e_S)_{S \in \binom{N_i}k}$; that is, $\bigwedge^k V_i = \bigwedge^k V \cap
  \bigwedge V_i$. 
  We claim that the left interior product behaves nicely with respect to
a colorful partition. To see this, we first need an auxiliary lemma about signs.

\begin{lemma}
\label{l:signs}
Let $U$ and $T$ be disjoint subsets of $N$ and for all $i\in [d]$ let $U_i := U \cap N_i$, $T_i := T
\cap N_i$, $u_i := |U_i|$ and $t_i := |T_i|$. Then
  \[
    \sgn(U, T) = (-1)^c \sgn(U_1, T_1) \cdots \sgn(U_d, T_d),
  \]
where $c$ depends only on $u_1, \dots, u_d$ and $t_1, \dots, t_d$.
\end{lemma}

\begin{proof}
  The value $\sgn(U,T)$ is $-1$ to the number of transpositions in the permutation
  $\pi$ of
  $U \cup T$ where we first place the elements of $U$ (in our given order on
  $N$) and then the elements of $T$ (in the same order). Considering that for $i < j$,
  $U_i$ precedes $U_j$ and $T_i$ precedes $T_j$, the order of the blocks $U_1,
  \dots, U_d, T_1, \dots, T_d$ in $\pi$ is 
  \[
    (U_1, \dots, U_d, T_1, \dots, T_d).
  \]
  After $c$ transpositions where $c$ depends only on $u_1,
  \dots, u_d, t_1, \dots, t_d$, we get a permutation $\pi'$ with the following
  order of blocks
  \[
    (U_1, T_1, U_2, T_2, \dots, U_d, T_d).
  \]
By the above, the sign of $\pi'$ equals $(-1)^c \sgn(U,T)$. On the other
  hand, as $T_i$ precedes $U_j$ for $i < j$ in our order on $N$, the sign of $\pi'$ is also equal the product $\sgn(U_1,T_1) \cdots \sgn(U_d,T_d)$.
  Equating these two expressions gives the desired identity.
\end{proof}

In the following proposition, the $f_i$ are not necessarily coming from a
colorful generic basis. However, we intend to apply it in this setting. With a
slight abuse of notation, we use $\bigwedge$ both for the exterior algebra as
well as for the wedge product of multiple elements. (This can be easily
distinguished from the context.) 

\begin{proposition}
\label{p:lip_colorful}
Suppose that $s_1, \dots, s_d$ and $t_1, \dots, t_d$ are nonnegative integers with $t_i \leq s_i
  \leq n_i$ for every $i \in [d]$. Suppose further that $f_i \in \bigwedge^{t_i} V_i$ and $h_i\in \bigwedge^{s_i} V_i$ for all $i\in [d]$. Then  
   $$\left (\bigwedge_{i=1}^d f_i \right ) \lip \left (\bigwedge_{i=1}^d h_i
    \right ) =
  \pm \bigwedge_{i=1}^d (f_i\lip h_i).$$
\end{proposition}

\begin{proof}
  We will show that 
  \begin{equation}
    \label{e:distribute_lip}
    \left (\bigwedge_{i=1}^d f_i \right ) \lip \left (\bigwedge_{i=1}^d h_i
    \right ) =
  (-1)^c \bigwedge_{i=1}^d (f_i\lip h_i)\end{equation}
where $c$ comes from Lemma~\ref{l:signs}; in particular, it depends only on
  $t_1, \dots, t_d$ and $s_1, \dots, s_d$.

  By bilinearity of~$\lip$ and $\wedge$ it is sufficient to
  prove~\eqref{e:distribute_lip} in the case when the $f_i$ and the  $h_i$ are basis elements of
  $\bigwedge^{t_i} V_i$ and $\bigwedge^{s_i} V_i$ respectively.
  So, assume for each $i\in [d]$ that $f_i = e_{T_i}$ and $h_i = e_{S_i}$ where $T_i \in
  \binom{N_i}{t_i}$ and $S_i \in \binom{N_i}{s_i}$, and let
  $T := T_1 \cup \dots \cup T_d$ and $S := S_1 \cup \dots \cup S_d$. 
  Then $\bigwedge_{i=1}^d f_i = e_T$ and $\bigwedge_{i=1}^d h_i = e_S$ by the
  definition of the exterior product $\wedge$. If $T_i \not\subseteq S_i$ for some $i
  \in [d]$, then $T \not \subseteq S$ and both sides
  of~\eqref{e:distribute_lip} vanish by Lemma~\ref{l:lip_fTfS}.
  Therefore, it remains to check the case that $T_i \subseteq S_i$ for every $i
  \in [d]$. Here by Lemma~\ref{l:signs} (with $U=S\setminus T$) and Lemma~\ref{l:lip_fTfS} we get
  \begin{align*}
    e_T \lip e_S &= \sgn(S\setminus T,T) e_{S\setminus T} \\
     &= 
    (-1)^c \sgn(S_1\setminus T_1, T_1) \cdots \sgn(S_d \setminus T_d, T_d)
    e_{S_1 \setminus T_1} \wedge \cdots \wedge e_{S_d \setminus T_d} \\
    &= 
    (-1)^c (e_{T_1} \lip e_{S_1}) \wedge \cdots \wedge (e_{T_d} \lip e_{S_d}),
  \end{align*}
as required.
\end{proof}

\section{Theorem~\ref{t:mwsat}: the lower bound}\label{sec:lower}
In this section we prove the lower bound in Theorem~\ref{t:mwsat}. Our proof
follows a strategy similar to~\cite{BBMR} and~\cite{Kal85}. Viewing the edges
of $K^q_{\nvec}$ as elements of the exterior algebra of $\R^{N}$, we will
define a linear mapping closely related to the weak saturation process and
lower-bound $\mwsat (K^q_{\nvec}, K^q_{\rvec})$ by the rank of the
corresponding matrix.

As outlined in Section~\ref{sec:algebra}, let $V$ be an $n$-dimensional real vector space with a basis $(e_v)_{v\in N}$, equipped with a standard inner product $\langle \cdot, \cdot \rangle$ with respect to this basis, that is,  $(e_v)_{v\in N}$ is orthonormal.  
Using the exterior product notation of Section~\ref{sec:algebra}, define
$$\SPAN K^q_{\nvec} := \SPAN
\{e_T \colon T \in E(K^q_{\nvec})\} \subseteq
\bigwedge\nolimits^q V.$$
\noindent
For an element $m\in \bigwedge^k V$ the
\emph{support of $m$} is the set
\begin{equation*}
  \supp(m)=\left \{S\in \binom{N}{k} \colon \langle e_S, m \rangle \neq 0 \right \}.
\end{equation*}
The following lemma, which converts the problem at hand into a constructive question in linear algebra, is analogous to Lemma 3 in~\cite{BBMR}.\footnote{Put equivalently in the language of~\cite{BBMR}, we map each edge of $K^q_n$ to vector in a certain vector space $\tilde{W}$, so that for each copy of $K^q_{\rvec}$ in $K^q_n$ the underlying vectors are linearly dependent with all coefficients involved being non-zero. This implies $\mwsat(K^q_n, K^q_{\rvec})\geq \dim \tilde{W}$. }
\begin{lemma}
  \label{l:rlboundd}
  Let $Y$ be a real vector space and~$\Gamma: \SPAN K^q_{\nvec} \rightarrow Y$ a linear map such that for every
  subset $R \subseteq N$ with $|R \cap
  N_i|=r_i$ for all $i\in [d]$ there exists an element $m\in \ker \Gamma$ with
  $\supp(m) = E(K^q_{\nvec}[R])$. Then $$\mwsat(K^q_{\nvec}, K^q_{\rvec})
  \geq \rank \Gamma.$$
\end{lemma}
\begin{proof}
  Suppose the $q$-graph $G_0$ is weakly $K^q_{\rvec}$-saturated in
  $K^q_{\nvec}$ and $|E(G_0)|=\mwsat(K^q_{\nvec}, K^q_{\rvec})$. Denote by
  $\{L_1,\dots ,L_k\}$ a corresponding saturating sequence and by $H_i$ a new copy of $K^q_{\rvec}$ that appears in $G_i = G_0 \cup
  \{L_1,\dots ,L_i\}$ with $L_i \in E(H_i)$. Let $Y_i = \SPAN \{\Gamma(e_T)
  \colon T\in E(G_i) \}$, and note that $Y_k=\Gamma(\SPAN K^q_{\nvec})$.
  By assumption, for
  each $i=1,\dots, k$ there exist non-zero coefficients $\{c_T: T\in E(H_i)\}$ such
  that  $\sum_{T\in E(H_i)} c_T \Gamma(e_T) = 0$. Therefore, $$\Gamma (e_{L_i}) =
  -\frac{1}{c_{L_i}} \sum_{T\in E(H_i)\setminus L_i} c_{T}\Gamma(e_T) \in Y_{i-1}.$$ We conclude
  that $Y_i = Y_{i-1}$.  By repeating this procedure we obtain 
  $$\mwsat(K^q_{\nvec}, K^q_{\rvec}) = |E(G_0)| \geq \dim Y_0 = \dim Y_k =
  \rank \Gamma.$$
\end{proof}
\noindent
Our goal now is to define a linear map $\Gamma$ as in Lemma~\ref{l:rlboundd}. For this
purpose let us fix an orthonormal colorful generic basis $(f_v)_{v\in N}$ of $V$
with respect to the partition of $N$, as described in Section~\ref{subsec:col}. Next, for each $i\in [d]$ choose
a set $J_i\subseteq N_i$ with $|J_i| = r_i - 1$ and a vertex $w_i\in
N_i\setminus J_i$. Put
$J:=\bigcup_{i\in [d]} J_i$ and $W := \{w_i
\colon i\in [d]\}$. Finally, set $s:=d-q$ and
\begin{equation}\label{e:defgs}
g :=
\sum_{T\in \binom{W}{s}} f_T.
\end{equation}
We can now state the following auxiliary lemma.
  \begin{lemma}
    \label{l:gsfZ}
    Let $z$ be an integer with $d \geq z \geq s$ and let $Z
    \in \binom{N}z$. Then
    \begin{enumerate}[(i)]
      \item $g \lip f_Z = 0$ if $|Z \cap W| < s$.
      \item If $z = s$, then $\langle g , f_Z \rangle = \begin{cases} \pm 1
      & \hbox{ if } Z \subseteq W, \\ 0 & \hbox{ if } Z \not\subseteq W.\end{cases}$
    \end{enumerate}
  \end{lemma}

\begin{proof}
  By~\eqref{e:defgs}, bilinearity of $\lip$, and Lemma~\ref{l:lip_fTfS} we get
  \begin{equation}
    \label{e:gs_fZ}
    g \lip f_Z = 
  \sum_{W'\in \binom{W}{s}}f_{W'}\lip f_Z
  = \sum_{W'\in \binom{W\cap Z}{s}} \pm f_{Z\setminus W'}.
  \end{equation}
The last expression is $0$ if $|Z\cap W|<s$; this shows (i). 

Now, assume that $z = s$. Then 
  \begin{equation}
    \label{e:pro_gsfZ}
  \langle g, f_Z \rangle = \langle f_\emptyset \wedge g, f_Z \rangle =
  \langle f_\emptyset, g \lip f_Z\rangle\stackrel{\eqref{e:gs_fZ}}{=}\sum_{W'\in \binom{W\cap Z}{s}} \pm \langle f_\emptyset, f_{Z\setminus W'}\rangle.
  \end{equation}
If $Z \not\subseteq W$, then $|Z \cap W| < z=s$, so $g \lip f_Z = 0$ from
(i), and thus~\eqref{e:pro_gsfZ} evaluates to $0$. On the other hand, if $Z \subseteq W$, then $\binom {W\cap Z}{s}=\{Z\}$. It follows that 
$$\langle g, f_Z \rangle\stackrel{\eqref{e:pro_gsfZ}}{=}\pm \langle f_{\emptyset},f_{\emptyset} \rangle=\pm 1,$$
yielding (ii).
\end{proof}

\noindent
We define the subspace
\begin{equation}\label{e:defU}
U := \SPAN \{g\lip f_T \colon T\in E(K^d_{\nvec}[N\setminus J]),
|T\cap W|\geq s\},
\end{equation}
and observe first that $U\subseteq \SPAN K^q_{\nvec}$. Indeed, for each $T$ in~\eqref{e:defU} and $W'\in \binom{W}{s}$, we have by Lemma~\ref{l:lip_fTfS} that $f_{W'}\lip f_T=0$ if $W'\not \subseteq T$ and $f_{W'}\lip f_T=\pm f_{T\setminus W'}$ if $W' \subseteq T$. In the latter case note that $T\setminus W'\in E(K^q_{\nvec})$, and the claim follows by bilinearity of $\lip$.

Let $Y$ be the orthogonal complement of $U$ in $\SPAN K^q_{\nvec}$ and let $\Gamma \colon \SPAN
K^q_{\nvec} \rightarrow \SPAN K^q_{\nvec}$ be the orthogonal projection on $Y$.
Our main technical lemma in this paper states that $\Gamma$ satisfies the assumptions of Lemma~\ref{l:rlboundd}.
\begin{lemma}
	\label{l:mpartite_coeff}
	Suppose that $R\subseteq N$ satisfies $|R\cap N_i| = r_i$ for every $i\in [d]$. Then,
	there exists $m \in \ker \Gamma$ such that $\supp(m) = E(K^q_{\nvec}[R])$.
\end{lemma}
\noindent
Deferring the proof of Lemma~\ref{l:mpartite_coeff}, let us first compute $\rank \Gamma$ and conclude the proof of Theorem~\ref{t:mwsat} assuming Lemma~\ref{l:mpartite_coeff}.

Notice that the sets $T\in
K^d_{\nvec}[N\setminus J]$ with $|T\cap W| \geq s$ are in bijective correspondence with the sets $T\setminus W \in K^p_{\nvec}[N\setminus (J\cup W)]$ with $p\leq q$. Using this bijection,
$$\dim U \stackrel{\eqref{e:defU}}{\leq} |\{T\in K^d_{\nvec}[N\setminus J]\colon |T\cap W| \geq s\}| = \sum_{\substack{I\subseteq [d]\\ |I|\leq q}} \prod_{i\in I} (n_i-r_i).$$ 
Consequently,
\begin{equation}\label{e:rankGamma} 
\rank \Gamma = \dim(\SPAN
K^q_{\nvec}) - \dim U \geq \sum_{I\in
  \binom{[d]}{q}}\prod_{i\in I}n_i - \sum_{\substack{I\subseteq [d] \\ |I|\leq
    q}} \prod_{i\in I}(n_i-r_i).
\end{equation}
\begin{proof}[Proof of Theorem~\ref{t:mwsat}]
  On the one hand, by Lemma~\ref{l:mpartite_coeff} the map $\Gamma$ satisfies the assumptions of Lemma~\ref{l:rlboundd}. Therefore,
  $$\mwsat(K^q_{\nvec},
  K^q_{\rvec}) \geq \rank \Gamma \stackrel{\eqref{e:rankGamma}}{\geq} \sum_{I\in
    \binom{[d]}{q}}\prod_{i\in I}n_i - \sum_{\substack{I\subseteq [d] \\ |I|\leq
      q}} \prod_{i\in I}(n_i-r_i).$$ On the other hand, Lemma~\ref{l:gubound}
  gives the same upper bound.
\end{proof}

\begin{proof}[Proof of Lemma~\ref{l:mpartite_coeff}]
  We claim that $$m = (g\wedge f_J) \lip e_R$$ is the desired
  element.\footnote{Let us briefly sketch the topological idea hidden behind this choice:
  As it can be easily deduced from the computations below, $m$ can be also
  expressed as $\pm g\lip \bigl( (f_{J_1} \lip e_{R_1}) \wedge \cdots \wedge (f_{J_d}
  \lip e_{R_d})\bigr)$. In the terminology of simplicial complexes interpreting
  loosely (i) $e_{R_i}$ as a full simplex on the vertex set $R_i$, (ii)
  $\wedge$  as a join of simplicial complexes and (iii) $\lip$ as an operator taking
  the skeleton of appropriate dimension, we gradually get the following:
  $f_{J_i} \lip e_{R_i}$ corresponds to the $0$-skeleton of the simplex on $R_i$,
  that is, the vertices of $R_i$. Then $(f_{J_1} \lip e_{R_1}) \wedge \cdots
  \wedge (f_{J_d}  \lip e_{R_d})$ corresponds to the join of the sets $R_i$,
  that is, the complete $d$-partite complex on $R_1, \dots, R_d$. Finally,
  applying $g\lip$ to this element takes the skeleton again reducing the
  dimension so that the corresponding hypergraph is the required $K^q_{\nvec}[R]$.}
  Let $R_i:=R\cap N_i$ for each $i\in [d]$. 
  
  First, we verify that $m\in \ker \Gamma =  U$. By
Proposition~\ref{p:lip_colorful}
we have $$f_J\lip e_R
= \pm (f_{J_1}\lip e_{R_1})\wedge \cdots \wedge (f_{J_d}\lip e_{R_d}).$$ By
Lemma~\ref{l:fSLeR} we can write each of these terms as \begin{equation}\label{e:lambdas}
f_{J_i}\lip e_{R_i} =
\sum_{v\in N_i\setminus J_i} \lambda_v f_v \ \ \text{with all} \ \ \lambda_v\neq 0.
\end{equation}
Combining this with~\eqref{e:wedge_multilinearity_new} gives
\begin{equation}
	\label{e:fJeR}
	f_J\lip e_R = \sum_{Z\in E(K^d_{\nvec}[N\setminus J])} \pm (\prod_{v\in Z}\lambda_v)
	f_Z.
\end{equation}
Therefore, we get
\begin{equation*}
	m
	= (g\wedge f_J) \lip e_R
	\stackrel{\eqref{e:double_left}}{=} g\lip (f_J \lip e_R)
	\stackrel{\eqref{e:fJeR}}{=} \sum_{Z\in E(K^d_{\nvec}[N\setminus J])} (\prod_{v\in Z}\lambda_v) g \lip f_Z
	= \sum_{\substack{Z\in E(K^d_{\nvec}[N\setminus J]) \\ |Z\cap W|\geq s}}
	(\prod_{v\in Z}\lambda_v) g\lip f_Z,
\end{equation*}
where the last equality follows by Lemma~\ref{l:gsfZ}(i) with $z = d$. Thus $m\in U$ as wanted.

  Next, we show that $\supp(m) = E(K^q_{\nvec}[R])$. As we just have shown, $m\in U \subseteq \SPAN K^q_{\nvec}$, i.e. $\supp(m) \subseteq E(K^q_{\nvec})$. Now, for $T \in E(K^q_{\nvec})$ we have
  \begin{equation}
    \label{e:eT_m}
    \langle e_T, m \rangle
    \stackrel{\eqref{e:lip_definition}}{=} \langle e_T\wedge( g\wedge f_{J}) , e_R \rangle
    = \pm  \langle (g\wedge f_J) \wedge e_T , e_R \rangle
   \stackrel{\eqref{e:lip_definition}}{=} 
   \pm  \langle g\wedge f_J , e_T \lip e_R \rangle.
  \end{equation}
If $T\notin
  E(K^q_{\nvec}[R])$, then $T\nsubseteq R$ and by Lemma~\ref{l:lip_fTfS} we have  $e_T\lip e_R = 0$, and consequently $\langle e_T, m \rangle = 0$. Hence, $T\notin \supp(m)$.

  Now assume that $T\in E(K^q_{\nvec}[R])$, i.e., $T\subseteq R$. By~\eqref{e:eT_m} and Lemma~\ref{l:lip_fTfS} we have 
\begin{equation}\label{e:etmgood}
\langle e_T, m \rangle {=} \pm  \langle g\wedge f_J , e_{R\setminus T} \rangle
\stackrel{\eqref{e:lip_definition}}{=} \pm \langle g, f_J \lip e_{R\setminus T} \rangle.
\end{equation}
Let $P := \{i\in [d]\colon T\cap N_i \neq \emptyset\}$ and
$P':=[d]\setminus P$. Using this notation we can write
$$e_{R\setminus T}
= \pm \left (\bigwedge_{i\in P}e_{R_i \setminus \tau_i} \right ) \wedge
 \left (\bigwedge_{i\in P'}e_{R_i} \right ),$$ where for each $i\in P$ the set $\tau_i = T\cap N_i$ contains a single vertex. 
  Applying Proposition~\ref{p:lip_colorful}, 
we
 deduce 
\begin{equation}\label{e:fjert}
f_J\lip e_{R\setminus T}
=\pm \left (\bigwedge_{i\in P} f_{J_i}\lip e_{R_i\setminus \tau_i} \right ) \wedge
\left (\bigwedge_{i\in P'} f_{J_i} \lip e_{R_i}\right ).
\end{equation} 
\noindent
  Since $|J_i|=r_i-1=|R_i\setminus \tau_i|$, by
  Lemma~\ref{l:lip_existence_uniqueness} for every $i \in P$ we have $f_{J_i}\lip e_{R_i\setminus \tau_i} \in \bigwedge^0 V$. Thus 
  $$f_{J_i}\lip e_{R_i\setminus \tau_i} = \langle e_\emptyset, f_{J_i}\lip
  e_{R_i\setminus \tau_i}\rangle e_\emptyset = \langle e_\emptyset \wedge f_{J_i},
  e_{R_i\setminus \tau_i}\rangle e_\emptyset = \langle f_{J_i},
  e_{R_i\setminus \tau_i}\rangle e_\emptyset,$$
  and notice that $\langle f_{J_i},
  e_{R_i\setminus \tau_i}\rangle\neq 0$ because $(f_v)_{v\in
    N_i}$ is generic with respect to $(e_v)_{v\in N_i}$.
Plugging it into~\eqref{e:fjert} yields
  \begin{equation}\label{e:gdfJeT}
   f_J \lip e_{R\setminus T} 
       =\pm \left (\bigwedge_{i\in P} \langle f_{J_i},
       e_{R_i\setminus \tau_i}\rangle e_\emptyset \right ) \wedge
   \left (\bigwedge_{i\in P'} f_{J_i} \lip e_{R_i}\right )
    =\pm \left (\prod_{i\in P} \langle f_{J_i}, e_{R_i\setminus \tau_i} \rangle \right ) \bigwedge_{i\in P'} f_{J_i} \lip e_{R_i}.
  \end{equation}

\noindent
Turning to $P'$, denote $N' := \bigcup_{i\in
    P'}N_i\setminus J_i$. We have
  \begin{equation}
    \label{e:fJieRi}
    \bigwedge_{i\in P'} f_{J_i} \lip e_{R_i}
    \stackrel{\eqref{e:lambdas}}{=}\bigwedge_{i\in P'} \left (\sum_{v\in N_i\setminus
      J_i}\lambda_v f_v \right )
   \stackrel{\eqref{e:wedge_multilinearity_new}}{=} \sum_{Z\in
    E(K^{s}_{\nvec}[N'])} \pm \left (\prod_{v\in Z} \lambda_v \right ) f_Z.
\end{equation}
 Therefore,
  \begin{equation}
    \label{e:final}
    \langle g, \bigwedge_{i\in P'} f_{J_i} \lip e_{R_i} \rangle
    = \sum_{Z\in E(K^{s}_{\nvec}[N'])} \pm (\prod_{v\in Z} \lambda_v) \langle
    g, f_Z \rangle = \pm \prod_{v\in W \cap N'} \lambda_v,
\end{equation}
where the second equality is due to Lemma~\ref{l:gsfZ}(ii), using that there is exactly one $Z \in E(K^{s}_{\nvec}[N'])$ with $Z \subseteq
W$, namely $Z = W \cap N'$.
Putting it all together,
\begin{align*}
  \langle e_T, m \rangle &\stackrel{\eqref{e:etmgood}}{=} \pm  \langle g, f_J \lip e_{R\setminus T} \rangle
  \stackrel{\eqref{e:gdfJeT}}{=} \pm (\prod_{i\in P} \langle f_{J_i}, e_{R_i\setminus \tau_i} \rangle) \langle g, \bigwedge_{i\in P'} f_{J_i} \lip e_{R_i} \rangle\\
  &\stackrel{\eqref{e:final}}{=} \pm (\prod_{i\in P} \langle f_{J_i}, e_{R_i\setminus \tau_i} \rangle ) \prod_{v\in W \cap N'} \lambda_v \neq 0,
\end{align*}
and consequently $T\in \supp(m)$.
\end{proof}

\section{Weak saturation in the clique}\label{sec:clique}

In this section we prove Theorem~\ref{t:tensor}. Let $H$ be a $q$-graph where
$q \geq 2$ without isolated vertices. We recall the notion of a \emph{link
hypergraph} of a vertex $v\in V(H)$: it is the $(q-1)$-graph (possibly with
isolated vertices) defined via 
$$L_H(v):=\{e\setminus \{v\}\colon e\in E(H), v\in e\}.
$$ 	
The \emph{co-degree} of a set $W$ of $q-1$ vertices in $H$ is
$$d_H(W):=|\{e\in E(H): W\subset e\}|.
$$
Define the \emph{minimum positive co-degree} of $H$, in notation $\delta^*(H)$, as
$$\delta^*(H):=\min\big\{d_H(W)\colon W\in \binom{V(H)}{q-1},d_H(W)>0 \big\}.$$
Notice that $\delta^*(H)\leq \delta^*(L_H(v))$ for all $v\in V(H)$, and equality holds for some $v$.  
\begin{lemma}\label{l:co-degree}
	$\wsat(n,H)\leq (\delta^*(H)-1)\binom{n}{q-1}+O_H(n^{q-2}).
	$
\end{lemma}
\begin{proof}
We apply induction on $q$. For $q=2$ this is a well-known fact (\cite{Faudree13}, Theorem
4).
Suppose now that $q\geq 3$ and the statement holds for all smaller values. Let $H$ be a $q$-graph and let $W=\{v_1,\dots, v_{q-1}\}$ be a set satisfying $d_H(W)=\delta^*(H)$. Let $H_1=L_H(v_1)$ be the link hypergraph of $v_1$, and observe that $\delta^*(H_1)=\delta^*(H)$. A weakly $H$-saturated $q$-graph on $[n]$ is obtained as follows.  Take a minimum weakly $H_1$-saturated $(q-1)$-graph on $[n-1]$ and insert $n$ into each edge; take a union of the resulting $q$-graph with a minimum weakly $H$-saturated $q$-graph on $[n-1]$. We therefore obtain
$$\wsat(n,H)\leq \wsat(n-1,H)+\wsat(n-1,H_1). 
$$
Iterating and applying the induction hypothesis,
\begin{align*}
\wsat(n,H)&\leq \wsat(|V(H)|,H)+\sum_{m=|V(H)|}^{n-1} \wsat(m,H_1)\\
&\leq (\delta^*(H_1)-1)\sum_{m=q-2}^{n-1}\binom{m}{q-2}+O_{H}(n^{q-2})\\
&=(\delta^*(H)-1)\binom{n}{q-1}+O_H(n^{q-2}).
\end{align*}
\end{proof}
The \emph{tensor product} of two $q$-graphs $G$ and $J$, $G\times
J$ is defined having the vertex set $V(G)\times V(J)$ and the edge set
$$E(G\times J)=\big\{\{(v_1,w_1),\dots (v_q,w_q)\}: \{v_1,\dots, v_q\}\in E(G), \{w_1,\dots , w_q\}\in E(J)\big \}.$$ 
(Note that every pair of edges in the original graphs produces $q!$ edges
in the product.)
\begin{lemma}\label{l:tensorpartite}
Let $H=K^{d}_{r_1,\dots,r_d}$, and let $F_n^d$ be the copy of $K^{d}(n;d)$ between the vertex sets $[n]\times\{1\},\dots,[n]\times\{d\}$. Then there exists a $d$-graph $E^d(n,H)\subseteq F_n^d\setminus (K^{d}_{[n]}\times K_{[d]}^{d})$ of size $O_H(n^{d-2})$ such that $$G(n,H):=(K^{d}_{[n]}\times K_{[d]}^{d})\sqcup E^d(n,H)$$ is weakly $H$-saturated in $F_n^d$. 
\end{lemma}
\begin{proof}
It suffices to prove the above statement when $r_1=\dots =r_d=:r$, i.e. when $H=K^d(r;d)$, as every edge creating a new copy of $K^d(\max\{r_1,\dots,r_d\};d)$ creates in particular a new copy of $K^{d}_{r_1,\dots,r_d}$.

We apply induction on $d$ and $n$. For $d=2$ and any $n\geq |V(H)|$ the graph $K_{[n]}\times K_{[2]}$ misses only a matching from $F_n^2$, making it already $H$-saturated in $F_n^2$, as can be easily checked. Moreover, for every fixed $H$ we can assume the statement to hold for all $n$ less than some large $C(H)$. 

For the induction step, fix $(n,d)$ and suppose that the statement holds for all
$(n',d')$ with $d'<d$ and all $(n'',d)$ with $n''< n$. It suffices to show that
$O_H(n^{d-3})$ edges can be added to $G(n-1,H)$ to satisfy the assertion; these
edges will be as follows. 

For each $i\in [d]$ let the $(d-1)$-graph $E'_i$ be an isomorphic copy of
  $E^{d-1}(n-1,K^{d-1}(r;d-1))$ between the sets $[n-1]\times \{j\}$ for $
  j\in[d]\setminus \{i\}$, such that $(K_{[n-1]}^{d-1}\times K_{[d]\setminus
  \{i\}}^{d-1})\sqcup E_i'$ is weakly $K^{d-1}(r;d-1)$-saturated in the complete $(d-1)$-partite $(d-1)$-graph between the sets $[n-1]\times \{j\}$ for $ j\in[d]\setminus \{i\}$. Let 
$$E_i:=\{e\sqcup \{(n,i)\}: e\in E_i'\}.$$ 
By the induction hypothesis $|E_i|=|E'_i|= O_H(n^{d-3})$.

Similarly, for each $\{i_1,i_2\}\in \binom{[d]}{2}$ apply Corollary~\ref{c:unordered} to obtain a $(d-2)$-graph $E'_{i_1,i_2}$ of size $O_H(n^{d-3})$ which is weakly
  $K^{d-2}(r;d-2)$-saturated in the copy of $K^{d-2}(n-1;d-2)$ between the sets $[n-1]\times
\{j\}$ for  $j\in[d]\setminus \{i_1,i_2\}$
(for $d=3$ take any $r-1$ vertices in $[n-1]\times [d]\setminus \{i_1,i_2\}$). As above, insert $(n,i_1)$ and $(n,i_2)$ into each edge of $E'_{i_1,i_2}$; let the resulting edge set be called $E_{i_1,i_2}$. 

Finally, take all edges of $F^d_n$ containing at least three vertices with $n$ as their first coordinate, and let $E_0$ be this edge set; clearly $|E_0|=O_H(n^{d-3})$ as well. Put
$$G(n,H):=G(n-1,H)\cup \bigcup_{i\in [d]}E_i\cup \bigcup_{\{i_1,i_2\}\in \binom{[d]}{2}}E_{i_1,i_2}\cup E_0,
$$
and 
$$E^d(n,H):=G(n,H)\setminus (K^{d}_{[n]}\times K_{[d]}^{d}).
$$

By the induction hypothesis and the bounds on the $|E_i|$, the $|E_{i_1,i_2}|$ and $|E_0|$, we have $|E^d(n,H)|=O_H(n^{d-2})$. To see that $G(n,H)$ is weakly
$H$-saturated, first note that by induction hypothesis $G(n-1,H)$ is weakly
$H$-saturated in $F^d_{n-1}$, hence the $d$-graph $G(n-1,H)\cup (K^{d}_{[n]}\times K_{[d]}^{d})\subseteq G(n,H)$ is weakly $H$-saturated in
$J_0:=F^d_{n-1}\cup(K^{d}_{[n]}\times K_{[d]}^{d})$.
Furthermore, let 
$$K_1:=\{e\in F_n^d\colon |e\cap (\{n\}\times [d])|=1\},$$
and 
$$K_2:=\{e\in F_n^d\colon |e\cap (\{n\}\times [d])|=2\}.$$
Let $J_1:=J_0\cup K_1$ and $J_2:=J_1\cup K_2$.
By construction, $J_0\cup \bigcup_{i\in [d]}E_i$ is weakly $H$-saturated in
$J_1$, $J_1\cup \bigcup_{\{i_1,i_2\}\in \binom{[d]}{2}}E_{i_1,i_2}$ is weakly
$H$-saturated in $J_2$ and $J_2\cup E_0=F_n^d$. Thus, $G(n,H)$ is weakly
$H$-saturated in $F_n^d$ as desired. This proves the induction step, and the
statement of the lemma follows.
\end{proof}
\begin{proof}[Proof of Theorem~\ref{t:tensor}]
For the first statement, suppose that $G\subseteq K_{n,n}$ is weakly $H$-saturated in $K_{n,n}$. Placing two $|V(H)|$-cliques on the parts of $G$ is easily seen to produce a weakly $H$-saturated graph in $K_{2n}$. Therefore,
\begin{equation}\label{e:biptrick1}
\wsat(2n,H)\leq \wsat(K_{n,n},H)+|V(H)|^2.
\end{equation}
Conversely, suppose that $G=G_0$ is weakly $H$-saturated in $K_{[n]}$ via a saturating sequence $e_1=\{i_1,j_1\},\dots, e_k=\{i_k,j_k\}$. For $1\leq \ell\leq k$ let $G_\ell=G_0\cup \{e_1,\dots e_\ell\}$, and let $H_\ell$ be a copy of $H$ in $G_\ell$ containing $e_\ell$.

Let $G^{bip}=G\times K_{[2]}$, i.e., $V(G^{bip})=[n]\times \{1,2\}$ and
  $$E(G^{bip})=\{\{(i,1),(j,2)\}: \{i,j\}\in E(G)\}.$$
We claim that 
$G^{bip}$ is weakly $H$-saturated in  $K_{[n]}^{bip}=K_{[n]}\times K_{[2]}$ via the $H$-saturating sequence 
\newline
$f_1, f'_1, \dots, f_k,f'_k$, where, for each $\ell \in [k]$, $f_\ell=\{(i_\ell,1),(j_\ell,2)\}$ and
$f'_\ell=\{(i_\ell,2),(j_\ell,1)\}$, and that 
 $G_{\ell-1}^{bip}\cup \{f_\ell,f'_\ell\}=G_{\ell}^{bip}$ for all $\ell\in [k]$ (where $G_{\ell}^{bip}$ is defined analogously, i.e., $G_{\ell}^{bip}=G_{\ell}\times K_{[2]}$). 
Indeed, let $(A,B)$ be a bipartition of $V(H_\ell)$ with $i_\ell\in A$ and
$j_\ell \in B$, and consider the analogous graph $H^b_\ell$ between $A\times
\{1\}$ and $B\times \{2\}$, i.e., for every $(i,j)\in A\times B$ we have
$\{(i,1),(j,2)\}\in E(H^b_\ell)$ if and only if $\{i,j\}\in E(H_\ell)$. Note that
$f_\ell\in E(H^b_\ell)$ is the only edge of $H^b_\ell$ not already present in
$G_{\ell-1}^{bip}$, therefore we can add it to the latter creating a new copy
  of $H$, namely $H^b_\ell$. Symmetrically, taking a graph $H'^b_\ell$ between
  $A\times \{2\}$ and $B\times \{1\}$ allows to add $f'_\ell$. Since
  $G_\ell=G_{\ell-1}\cup e_\ell$, we have $G_{\ell-1}^{bip}\cup \{f_\ell,f'_\ell\}=G_{\ell}^{bip}$.
Finally, note that $G^{bip}\cup \{f_1,\dots,f'_k\}=G_k^{bip}=K_{[n]}^{bip}$.

Note that $K_{[n]}^{bip}$ is isomorphic to $K_{n,n}$ minus a perfect matching,
  and it is a straightforward check that this graph is $H$-saturated in $K_{n,n}$ (we can assume that $|V(H)|\leq n$). We have thus shown 
\begin{equation}\label{e:biptrick2}
\wsat(K_{n,n},H)\leq 2 \wsat(n,H).
\end{equation}
Combining~\eqref{e:biptrick1} and~\eqref{e:biptrick2} gives
$$\frac{\wsat(2n,H)}{2n}-o(1)\leq \frac{\wsat(K_{n,n},H)}{2n}\leq \frac{\wsat(n,H)}{n},
$$
and taking the limit,~\eqref{e:limconst} follows readily.

For the second statement, denote $H=K^{d}_{r_1,\dots,r_d}$ where $1\leq r_1\leq\dots\leq r_d$. Observe that the upper bound in~\eqref{e:hypreduction} holds by Lemma~\ref{l:co-degree}, as $\delta^*(H)=r_1$. To prove the lower bound, suppose $G$ is weakly $H$-saturated in $K^{d}_{[n]}$, and that $|E(G)|=\wsat(n,H)$ . Let $G^{mult}=G\times K_{[d]}^{d}$, that is,
$V(G^{mult})=[n]\times [d]$ and 
$$E(G^{mult})=\{\{(i_1,1),\dots, (i_d,d)\}:\{i_1,\dots, i_d\} \in E(G)\}.
$$

Essentially the same argument as for $G^{bip}$ before shows that $G^{mult}$ is weakly $H$-saturated in $K^{d}_{[n]}\times K_{[d]}^{d}$. By Lemma~\ref{l:tensorpartite} adding further $O_H(n^{d-2})$ edges creates a weakly $H$-saturated $d$-graph in $K^d(n;d)$. Hence,
\begin{equation}\label{eq:multitensor}
\wsat(K^{d}(n;d),H) \leq |E(G^{mult})|+O(n^{d-2}) = d! \wsat(n,H)+O(n^{d-2}).
\end{equation}
On the other hand, Moshkovitz and Shapira~\cite{MoshkShap} proved that $\wsat(K^{d}(n;d),H)=d(r_1-1)n^{d-1}+O(n^{d-2})$. Combining this with~\eqref{eq:multitensor} yields the lower bound in~\eqref{e:hypreduction}. 
\end{proof}

\bibliographystyle{alpha}
\bibliography{wsat}
 
\end{document}